\documentclass[a4paper,12pt]{article}

\usepackage{amssymb,amsmath,amsthm,amsfonts,latexsym}
\usepackage{graphicx}

\theoremstyle{plain}
\newtheorem{Thm}{Theorem}[section]
\newtheorem{Lem}[Thm]{Lemma}
\newtheorem{Prop}[Thm]{Proposition}

\theoremstyle{definition}
\newtheorem{Exa}[Thm]{Example}
\newtheorem*{Rem}{Remark}
\newtheorem*{Defi}{Definition}

\title{T-partition systems and travel groupoids on a graph}

\author{
{\sc Jung Rae CHO}
\qquad {\sc Jeongmi PARK}\\ 
[1ex]
{\small Department of Mathematics, 
Pusan National University, 
Busan 609-735, Korea} \\
{\small \texttt{jungcho@pusan.ac.kr}} 
\qquad 
{\small \texttt{jm1015@pusan.ac.kr}} \\
\\
{\sc Yoshio SANO}
\thanks{This work was supported by JSPS KAKENHI Grant Number 15K20885.}\\
[1ex]
{\small Division of Information Engineering}\\
{\small Faculty of Engineering, Information and Systems}\\
{\small University of Tsukuba, 
Ibaraki 305-8573, Japan}\\
{\small \texttt{sano@cs.tsukuba.ac.jp}}
}

\date{}

\begin{document}

\maketitle


\begin{abstract}
The notion of travel groupoids 
was introduced by L. Nebesk{\'y} in 2006 
in connection with a study on geodetic graphs. 
A travel groupoid is a pair of a set $V$ and 
a binary operation $*$ on $V$ 
satisfying two axioms. 
For a travel groupoid, we can associate a graph. 
We say that a graph $G$ has a travel groupoid 
if the graph associated with the travel groupoid 
is equal to $G$. 
Nebesk{\'y} gave a characterization 
for finite graphs to have a travel groupoid. 

In this paper, 
we introduce the notion of T-partition systems on a graph  
and give a characterization of travel groupoids on a graph 
in terms of T-partition systems. 
\end{abstract}





\section{Introduction}

All graphs in this paper 
are (finite or infinite) undirected graphs 
with no loops and no multiple edges. 

A \emph{groupoid} is the pair $(V,*)$ 
of a nonempty set $V$ and a binary operation $*$ on $V$. 
The notion of travel groupoids was introduced 
by L. Nebesk{\'y} \cite{TG2006} in 2006. 
First, let us recall the definition of travel groupoids. 

A \emph{travel groupoid} is a groupoid $(V,*)$ 
satisfying the following axioms (t1) and (t2):
\begin{itemize}
\item[(t1)] 
$(u * v) * u = u$ (for all $u,v \in V$), 
\item[(t2)] 
if $(u * v) * v = u$, then $u = v$ (for all $u,v \in V$).
\end{itemize}
A travel groupoid is said to be \emph{simple} 
if the following condition holds. 
\begin{itemize}
\item[{\rm (t3)}]
If $v*u \neq u$, then $u*(v*u)=u*v$ (for any $u,v \in V$). 
\end{itemize}
A travel groupoid is said to be \emph{smooth} 
if the following condition holds. 
\begin{itemize}
\item[{\rm (t4)}]
If $u*v=u*w$, then $u*(v*w)=u*v$ (for any $u,v,w \in V$). 
\end{itemize}
A travel groupoid is said to be \emph{semi-smooth} 
if the following condition holds. 
\begin{itemize}
\item[{\rm (t5)}]
If $u*v=u*w$, then $u*(v*w)=u*v$ or $u*((v*w)*w)=u*v$ 
(for any $u,v,w \in V$). 
\end{itemize}

Let $(V,*)$ be a travel groupoid, 
and let $G$ be a graph. 
We say that \emph{$(V,*)$ is on $G$}
or that \emph{$G$ has $(V,*)$} if $V(G) = V$ and
$E(G) = \{\{u,v \} \mid u, v \in V, u \neq v, 
\text{ and } u * v = v \}$. 
Note that every travel groupoid is on exactly one graph.

In this paper, we introduce the notion of 
T-partition systems on a graph  
and give a characterization of travel groupoids 
on a graph. 
This paper is organized as follows: 
In Section 2, we define the right translation system of a groupoid 
and characterize travel groupoids 
in terms of the right translation systems of groupoids. 
Section 3 is the main part of this paper. 
We introduce T-partition systems on a graph  
and give a characterization of travel groupoids on a graph 
in terms of T-partition systems. 
In Section 4, we consider 
simple, smooth, and semi-smooth travel groupoids, 
and give characterizations for them.

\section{The right translation system of a travel groupoid}

\begin{Defi}
Let $(V,*)$ be a groupoid. 
For $u,v \in V$, let 
$V^R_{u,v}$ be 
the set of elements whose right translations send the element $u$ 
to the element $v$, i.e., 
\[
V_{u,v} = V^R_{u,v} := \{ w \in V \mid u * w = v \}. 
\]
Then we call the system $(V_{u,v} \mid (u,v) \in V \times V)$
the \emph{right translation system} of the groupoid $(V,*)$. 
\qed
\end{Defi}

\begin{Rem}
For each element $w$ in a groupoid $(V,*)$, 
the \emph{right translation map} $R_w:V \to V$  
is defined by $R_w(u):=u*w$ for $u \in V$. 
Then the set $V^R_{u,v}$ is given by $V^R_{u,v} = \{ w \in V \mid R_w(u)=v \}$. 

We can also consider the set $V^R_{u,v}$ defined above 
as the inverse image of $\{v\}$ 
through the left translation map $L_u:V \to V$ of the groupoid $(V,*)$ 
defined by $L_u(w):=u*w$ for $w \in V$, i.e., 
$V^R_{u,v} = L_u^{-1}(\{v\})$. 
\end{Rem}

\begin{Lem}\label{lem:s1}
Let $(V,*)$ be a groupoid. 
Then, 
$(V,*)$ satisfies the condition {\rm (t1)}  
if and only if 
the right translation system $(V_{u,v} \mid (u,v) \in V \times V)$ 
of $(V,*)$ 
satisfies the following property: 
\begin{itemize}
\item[{\rm (R1)}]
For any $u, v \in V$, 
if $V_{u,v} \neq \emptyset$, then $u \in V_{v,u}$. 
\end{itemize}
\end{Lem}

\begin{proof}
Let $(V,*)$ be a groupoid satisfying the condition {\rm (t1)}. 
Take any $u,v \in V$ such that $V_{u,v} \neq \emptyset$. 
Take an element $x \in V_{u,v}$. 
Then $u*x=v$. 
By (t1), we have $(u*x)*u=u$. 
Therefore, we have $v*u=u$ and so $u \in V_{v,u}$. 
Thus the property (R1) holds. 

Suppose that 
the right translation system $(V_{u,v} \mid (u,v) \in V \times V)$ 
of a groupoid $(V,*)$ satisfies the property (R1). 
Take any $u,v \in V$. 
Then there exists a unique element $x \in V$ such that $v \in V_{u,x}$. 
Since $V_{u,x} \neq \emptyset$,
it follows from the property (R1) that $u \in V_{x,u}$ and so $x*u=u$. 
Since $v \in V_{u,x}$, we have $u*v=x$.
Therefore, we have $(u*v)*u=u$. 
Thus $(V,*)$ satisfies the condition (t1). 
\end{proof}

\begin{Lem}\label{lem:s2}
Let $(V,*)$ be a groupoid. 
Then, 
$(V,*)$ satisfies the condition {\rm (t2)} 
if and only if 
the right translation system $(V_{u,v} \mid (u,v) \in V \times V)$ 
of $(V,*)$ 
satisfies the following property: 
\begin{itemize}
\item[{\rm (R2)}]
For any $u, v \in V$ with $u \neq v$, 
$V_{u,v} \cap V_{v,u} = \emptyset$. 
\end{itemize}
\end{Lem}

\begin{proof} 
Let $(V,*)$ be a groupoid satisfying the condition {\rm (t2)}. 
Suppose that $V_{u,v} \cap V_{v,u} \neq \emptyset$ 
for some $u,v \in V$ with $u \neq v$. 
Take an element $x \in V_{u,v} \cap V_{v,u}$. 
Then $u*x=v$ and $v*x=u$. 
Therefore, we have $(u*x)*x=u$, 
which is a contradiction to the property (t2). 
Thus the property (R2) holds. 

Suppose that 
the right translation system $(V_{u,v} \mid (u,v) \in V \times V)$ 
of a groupoid $(V,*)$ satisfies the property (R2). 
Suppose that 
there exist $u,v \in V$ with $u \neq v$ 
such that $(u*v)*v=u$. 
Let $x:=u*v$. Then $v \in V_{u,x}$. 
Since $(u*v)*v=u$, we have $x * v = u$ and so $v \in V_{x,u}$. 
Therefore $v \in V_{u,x} \cap V_{x,u}$, 
which is a contradiction to the property (R2). 
Thus $(V,*)$ satisfies the condition (t2). 
\end{proof}

The following gives a characterization of travel groupoids 
in terms of the right translation systems of groupoids. 

\begin{Thm}\label{thm:main0}
Let $(V,*)$ be a groupoid. 
Then, 
$(V,*)$ is a travel groupoid 
if and only if 
the right translation system of $(V,*)$ 
satisfies the properties {\rm (R1)} and {\rm (R2)}. 
\end{Thm}

\begin{proof}
The theorem follows from Lemmas \ref{lem:s1} and \ref{lem:s2}.
\end{proof}

\section{T-partition systems on a graph}

We introduce the notion of T-partition systems on a graph 
to characterize travel groupoids on a graph. 

For a vertex $u$ in a graph $G$, 
let $N_G[u]$ denote the closed neighborhood of $u$ in $G$, i.e., 
$N_G[u] := \{u\} \cup \{v \in V \mid \{u,v\}\in E \}$. 

\begin{Defi}
Let $G=(V,E)$ be a graph. 
A \emph{T-partition system} on $G$ 
is a system 
$$\mathcal{P} = (V_{u,v} \subseteq V \mid (u,v) \in V \times V)$$ 
satisfying the following conditions: 
\begin{itemize}
\item[{\rm (P0)}] 
$\mathcal{P}_u := \{V_{u,v} \mid v \in N_G[u] \}$ is a partition of $V$ 
\quad 
(for any $u \in V$); 
\item[{\rm (P1a)}]
$V_{u,u} = \{u\}$ 
\quad 
(for any $u \in V$); 
\item[{\rm (P1b)}]
$v \in V_{u,v}$ 
$\iff$ 
$\{u, v\} \in E$ 
\quad
(for any $u, v \in V$ with $u \neq v$); 
\item[{\rm (P1c)}]
$V_{u,v} = \emptyset$ 
$\iff$  
$\{u, v\} \not\in E$ 
\quad
(for any $u, v \in V$ with $u \neq v$); 
\item[{\rm (P2)}]
$V_{u,v} \cap V_{v,u} = \emptyset$ 
\quad
(for any $u, v \in V$ with $u \neq v$). 
\qed
\end{itemize}
\end{Defi}

\begin{Rem}
Let $\mathcal{P}=(V_{u,v} \mid (u,v) \in V \times V)$
be a T-partition system on a graph $G$. 
It follows from the conditions (P0) and (P1c) that,
for any $u, v \in V$, 
there exists the unique vertex $w$ such that $v \in V_{u,w}$. 
\end{Rem}

\begin{Defi}
Let $\mathcal{P}=(V_{u,v} \mid (u,v) \in V \times V)$
be a T-partition system on a graph $G$. 
For $u,v \in V$, 
let $f_u(v)$ be the unique vertex $w$ such that $v \in V_{u,w}$. 
We define a binary operation $*$ on $V$ by 
$u*v:=f_u(v)$. 
We call $(V,*)$ the \emph{groupoid associated with} $\mathcal{P}$. 
\qed
\end{Defi}

\begin{Rem}
It follows from definitions that the right translation system of the
groupoid associated with a T-partition system $\mathcal{P}$ on a graph $G$ 
is the same as $\mathcal{P}$. 
\end{Rem}

\begin{Lem}\label{lem:VPS2TG}
Let $G$ be a graph, and 
let $\mathcal{P}$
be a T-partition system on $G$. 
Then, the groupoid associated with $\mathcal{P}$ 
is a travel groupoid on $G$. 
\end{Lem}

\begin{proof}
It follows from 
the properties (P1a), (P1b), and (P1c) in 
the definition 
a T-partition system on a graph 
that a T-partition system 
$\mathcal{P}=(V_{u,v} \mid (u,v) \in V \times V)$ satisfies 
the property (R1). 
Moreover, the condition (P2) is the same as the property (R2). 
By Theorem \ref{thm:main0}, 
$(V,*)$ is a travel groupoid on $G$. 

Now we show that $(V,*)$ is on the graph $G$. 
Let $G_{(V,*)}$ be the graph which has $(V,*)$. 
We show that $G_{(V,*)}=G$. 
Take any edge $\{u,v\}$ in $G_{(V,*)}$. 
Then, we have $u*v=v$. Therefore $v \in V_{u,v}$. 
Thus $\{u,v\}$ is an edge in $G$, and so $E(G_{(V,*)}) \subseteq E(G)$. 
Take any edge $\{u,v\}$ in $G$. 
Then, we have $v \in V_{u,v}$. 
Therefore $u*v=v$. Thus $\{u,v\}$ is an edge in $G_{(V,*)}$, 
and so $E(G) \subseteq E(G_{(V,*)})$. 
Hence we have $G_{(V,*)}=G$. 
\end{proof}

\begin{Lem}\label{lem:TG2VPS}
Let $G$ be a graph, and 
let $(V,*)$ be a travel groupoid on $G$. 
Then, the right translation system 
of $(V,*)$ 
is a T-partition system on $G$. 
\end{Lem}

\begin{proof}
Let $(V_{u,v} \mid (u,v) \in V \times V)$ be 
the right translation system 
of a travel groupoid $(V,*)$ on a graph $G$. 

Fix any $u \in V$. 
Since $u*w$ is in $N_G[u]$ for any $w \in V$, 
we have $\bigcup_{v \in N_G[u]} V_{u,v} = V$. 
Suppose that $V_{u,x} \cap V_{u,y} \neq \emptyset$
for some $x, y \in N_G[u]$ with $x \neq y$. 
Take $z \in V_{u,x} \cap V_{u,y}$. 
Then $u*z=x$ and $u*z=y$. 
Therefore we have $x=y$, which is a contradiction to 
the assumption $x \neq y$. 
Thus $V_{u,x} \cap V_{u,y} = \emptyset$
for any $x, y \in N_G[u]$ with $x \neq y$. 
Therefore 
$\{ V_{u,v} \mid v \in N_G[u] \}$ is a partition of $V$. 
Thus the condition (P0) holds. 

By \cite[Proposition 2 (2)]{TG2006}, 
$u*v=u$ if and only if $u=v$. 
Therefore $V_{u,u} = \{u\}$ and thus the condition (P1a) holds. 

If $u$ and $v$ are adjacent in $G$, then we have $u*v=v$, 
and so $v \in V_{u,v}$. 
If $v \in V_{u,v}$, then then we have $u*v=v$, 
and so $\{u, v\}$ is an edge of $G$. 
Thus the condition (P1b) holds.  

Since $u * w$ is a neighbor of $u$ in $G$ for any $w \in V \setminus \{u\}$, 
if $u$ and $v$ are not adjacent in $G$, 
then $V_{u,v} = \emptyset$. 
If $V_{u,v} = \emptyset$, then we have $u*v \neq v$, 
and so $u$ and $v$ are not adjacent in $G$. 
Thus the condition (P1c) holds.

Since $(V,*)$ satisfies the condition (t2), 
it follows from Lemma \ref{lem:s2} that 
the condition (P2) holds. 

Hence the right translation system of $(V,*)$ 
is a T-partition system on $G$. 
\end{proof}

The following gives a characterization of travel groupoids on a graph 
in terms of T-partition systems on the graph. 

\begin{Thm}\label{thm:main1}
Let $G$ be a graph. 
Then, there exists a one-to-one correspondence between 
the set of all travel groupoids on $G$ 
and the set of all T-partition systems on $G$.  
\end{Thm}

\begin{proof}
Let $V:=V(G)$. 
Let $\mathrm{TG}(G)$ denote the set of all travel groupoids on $G$ 
and let $\mathrm{TPS}(G)$ 
denote the set of all T-partition systems on $G$. 

We define a map $\Phi:\mathrm{TPS}(G) \to \mathrm{TG}(G)$ as follows: 
For 
$\mathcal{P} 
\in \mathrm{TPS}(G)$, 
let $\Phi(\mathcal{P})$ be the groupoid 
associated with $\mathcal{P}$. 
By Lemma \ref{lem:VPS2TG}, 
$\Phi(\mathcal{P})$ 
is a travel groupoid on $G$. 

We define a map $\Psi:\mathrm{TG}(G) \to \mathrm{TPS}(G)$ as follows: 
For 
$(V,*) \in \mathrm{TG}(G)$, 
let $\Psi((V,*))$ be the right translation system of $(V,*)$. 
By Lemma \ref{lem:TG2VPS}, 
$\Psi((V,*))$ is a T-partition system on $G$. 

Then, we can check that	$\Psi(\Phi(\mathcal{P})) = \mathcal{P}$ holds for any 
$\mathcal{P} \in \mathrm{TPS}(G)$
and that $\Phi(\Psi((V,*))) = (V,*)$ holds for any $(V,*) \in\mathrm{TG}(G)$. 
Hence the map $\Phi$ is a one-to-one correspondence between 
the sets $\mathrm{TPS}(G)$ and $\mathrm{TG}(G)$.
\end{proof}

\begin{Exa}\label{ex:TGon4cycle}
Let $G=(V,E)$ be the graph defined by 
$V=\{a,b,c,d\}$ and 
$E=\{\{a, b\},\{b, c\},\{c, d\},\{a, d\} \}$. 
Let $(V,*)$ be the groupoid defined by 
\[
\begin{array}{llll}
a * a = a, &  a * b = b, &  a * c = d, &  a * d = d, \\
b * a = a, &  b * b = b, &  b * c = c, &  b * d = a, \\
c * a = b, &  c * b = b, &  c * c = c, &  c * d = d, \\
d * a = a, &  d * b = c, &  d * c = c, &  d * d = d. \\
\end{array}
\]
Then $(V,*)$ is a travel groupoid on the graph $G$. 

Let $\mathcal{P}=(V_{u,v} \mid (u,v) \in V \times V)$
be the system of vertex subsets defined by 
\[
\begin{array}{llll}
V_{a,a} = \{a\}, &  V_{a,b} = \{b\}, &  V_{a,c} = \emptyset, &  V_{a,d} = \{c,d\}, \\
V_{b,a} = \{a,d\}, &  V_{b,b} = \{b\}, &  V_{b,c} = \{c\}, &  V_{b,d} = \emptyset, \\
V_{c,a} = \emptyset, &  V_{c,b} = \{a,b\}, &  V_{c,c} = \{c\}, &  V_{c,d} = \{d\}, \\
V_{d,a} = \{a\}, &  V_{d,b} = \emptyset, &  V_{d,c} = \{b,c\}, &  V_{d,d} = \{d\}. \\
\end{array}
\]
Then $\mathcal{P}$ is a T-partition system on $G$. 

Now we can see that the right translation system of $(V,*)$ is $\mathcal{P}$ 
and that the groupoid associated with $\mathcal{P}$ is $(V,*)$. 
\qed
\end{Exa}

\section{Simple, smooth, and semi-smooth systems}

In this section, we give characterizations 
of simple, smooth, and semi-smooth travel groupoids on a graph 
in terms of T-partition systems on the graph. 

\subsection{Simple T-partition systems}

\begin{Prop}\label{prop:simple-chara}
Let $(V,*)$ be a travel groupoid. 
For $u,v \in V$, let 
\[
V_{u,v} := \{ w \in V \mid u * w = v \}. 
\]
Then, the following conditions are equivalent: 
\begin{itemize}
\item[{\rm (a)}]
$(V,*)$ is simple; 
\item[{\rm (b)}]
For $u,v,x,y \in V$, if  
$u \in V_{v,x}$, $v \in V_{u,y}$, $u \neq x$, and $v \neq y$, 
then
$x \in V_{u,y}$ and $y \in V_{v,x}$. 
\end{itemize}
\end{Prop}

\begin{proof}
First, we show that (a) implies (b). 
Take any $u,v,x,y \in V$ 
such that 
$u \in V_{v,x}$, $v \in V_{u,y}$, $u \neq x$, $v \neq y$ 
Then $v*u = x \neq u$ and $u*v = y \neq v$. 
Since $(V,*)$ is simple by (a), 
$v*u \neq u$ implies $u * (v * u) = u * v$ 
and 
$u*v \neq v$ implies $v * (u * v) = v * u$. 
Therefore, we have 
$u*x=y$ and $v*y=x$. 
Thus 
$x \in V_{u,y}$ and $y \in V_{v,x}$.  

Second, we show that (b) implies (a). 
Take any $u,v \in V$ such that $v*u \neq u$. 
Note that $v*u \neq u$ implies $u*v \neq v$ (cf. \cite[Proposition 2 (1)]{TG2006}).  
Let $x:=v*u$. Then $u \in V_{v,x}$ and $u \neq x$. 
Let $y:=u*v$. Then $v \in V_{u,y}$ and $v \neq y$.  
By (b), we have $x \in V_{u,y}$ and $y \in V_{v,x}$ 
Therefore $u*x=y$, that is, $u*(v*u)=u*v$. 
Thus $(V,*)$ is simple. 
\end{proof}

\begin{Defi}
A T-partition system $(V_{u,v} \mid (u,v) \in V \times V)$ on a graph $G$ 
is said to be \emph{simple} 
if the following condition holds: 
\begin{itemize}
\item[{\rm (R3)}]
For $u,v,x,y \in V$, 
if $u \in V_{v,x}$, $v \in V_{u,y}$, $u \neq x$, and $v \neq y$, 
then $x \in V_{u,y}$ and $y \in V_{v,x}$. 
\qed
\end{itemize}
\end{Defi}

\begin{Lem}\label{lem:SiVPS2SiTG}
Let $G$ be a graph, and 
let $\mathcal{P}$
be a simple T-partition system on $G$. 
Then, the groupoid associated with $\mathcal{P}$ 
is a simple travel groupoid on $G$. 
\end{Lem}

\begin{proof}
The lemma follows from Lemma \ref{lem:VPS2TG} and 
Proposition \ref{prop:simple-chara}. 
\end{proof}

\begin{Lem}\label{lem:SiTG2SiVPS}
Let $G$ be a graph, and 
let $(V,*)$ be a simple travel groupoid on $G$. 
Then, the right translation system 
of $(V,*)$ 
is a simple T-partition system on $G$. 
\end{Lem}

\begin{proof}
The lemma follows from Lemma \ref{lem:TG2VPS} 
and Proposition \ref{prop:simple-chara}. 
\end{proof}

\begin{Thm}\label{thm:chara-simple}
Let $G$ be a graph. 
Then, there exists a one-to-one correspondence between 
the set of all smooth travel groupoids on $G$ 
and the set of all smooth T-partition systems on $G$.  
\end{Thm}

\begin{proof}
The theorem follows from Theorem \ref{thm:main1} 
and Lemmas \ref{lem:SiVPS2SiTG} and \ref{lem:SiTG2SiVPS}. 
\end{proof}

\subsection{Smooth T-partition systems}

\begin{Prop}\label{prop:smooth-chara}
Let $(V,*)$ be a travel groupoid. 
For $u,v \in V$, let 
\[
V_{u,v} := \{ w \in V \mid u * w = v \}. 
\]
Then, the following conditions are equivalent: 
\begin{itemize}
\item[{\rm (a)}]
$(V,*)$ is smooth; 
\item[{\rm (b)}]
For $u,v,x,y \in V$, 
if $x,y \in V_{u,v}$, then $x*y \in V_{u,v}$; 
\item[{\rm (c)}]
For $u,v,x,y,z \in V$, 
if $x, y \in V_{u,v}$ and $x \in V_{y,z}$, then $z \in V_{u,v}$. 
\end{itemize}
\end{Prop}

\begin{proof}
First, we show that (a) implies (b). 
Take any $u,v,x,y \in V$ 
such that $x,y \in V_{u,v}$. 
Then $u*x=v$ and $u*y=v$, so $u*x=u*y$. 
Since $(V,*)$ is smooth by (a), 
$u*x=u*y$ implies $u*(x*y)=u*x=v$. 
Thus $x*y \in V_{u,v}$. 

Second, we show that (b) implies (c). 
Take any $u,v,x,y,z \in V$ such that 
$x,y \in V_{u,v}$ and $y \in V_{x,z}$. 
Then $x*y=z$. 
By (b), we have $x*y \in V_{u,v}$. 
Thus $z \in V_{u,v}$. 

Third, we show that (c) implies (a). 
Take any $u,x,y \in V$
such that $u*x=u*y$. 
Let $v:=u*x=u*y$. Then $x,y \in V_{u,v}$. 
Let $z:=x*y$. Then $y \in V_{x,z}$. 
By (c), we have $z \in V_{u,v}$. 
Therefore $u*z=v$, that is, $u*(x*y)=u*x$. 
Thus $(V,*)$ is smooth. 
\end{proof}

\begin{Defi}
A T-partition system $(V_{u,v} \mid (u,v) \in V \times V)$ on a graph $G$ 
is said to be \emph{smooth} 
if the following condition holds: 
\begin{itemize}
\item[{\rm (R4)}]
For $u,v,x,y,z \in V$, 
if $x, y \in V_{u,v}$ and $x \in V_{y,z}$, then $z \in V_{u,v}$. 
\qed
\end{itemize}
\end{Defi}

\begin{Lem}\label{lem:SmVPS2SmTG}
Let $G$ be a graph, and 
let $\mathcal{P}$
be a smooth T-partition system on $G$. 
Then, the groupoid associated with $\mathcal{P}$ 
is a smooth travel groupoid on $G$. 
\end{Lem}

\begin{proof}
The lemma follows from Lemma \ref{lem:VPS2TG} 
and Proposition \ref{prop:smooth-chara}. 
\end{proof}

\begin{Lem}\label{lem:SmTG2SmVPS}
Let $G$ be a graph, and 
let $(V,*)$ be a smooth travel groupoid on $G$. 
Then, the right translation system 
of $(V,*)$ 
is a smooth T-partition system on $G$. 
\end{Lem}

\begin{proof}
The lemma follows from Lemma \ref{lem:TG2VPS} 
and Proposition \ref{prop:smooth-chara}. 
\end{proof}

\begin{Thm}\label{thm:chara-smooth}
Let $G$ be a graph. 
Then, there exists a one-to-one correspondence between 
the set of all smooth travel groupoids on $G$ 
and the set of all smooth T-partition systems on $G$.  
\end{Thm}

\begin{proof}
The theorem follows from Theorem \ref{thm:main1} 
and Lemmas \ref{lem:SmVPS2SmTG} and \ref{lem:SmTG2SmVPS}. 
\end{proof}

\begin{Rem}
Matsumoto and Mizusawa \cite{MM} gave an algorithmic way to construct 
a smooth travel groupoid on a finite connected graph. 
\qed 
\end{Rem}

\subsection{Semi-smooth T-partition systems}

\begin{Prop}\label{prop:semismooth-chara}
Let $(V,*)$ be a travel groupoid. 
For $u,v \in V$, let 
\[
V_{u,v} := \{ w \in V \mid u * w = v \}. 
\]
Then, the following conditions are equivalent: 
\begin{itemize}
\item[{\rm (a)}]
$(V,*)$ is semi-smooth; 
\item[{\rm (b)}]
For $u,v,x,y \in V$, 
if $x,y \in V_{u,v}$, then $x*y \in V_{u,v}$ or $x*^2y \in V_{u,v}$; 
\item[{\rm (c)}]
For $u,v,x,y,z,w \in V$, 
if $x, y \in V_{u,v}$ and $x \in V_{y,z} \cap V_{z,w}$, 
then $z \in V_{u,v}$ or $w \in V_{u,v}$, 
\end{itemize}
where $x*^2y := (x*y)*y$. 
\end{Prop}

\begin{proof}
First, we show that (a) implies (b). 
Take any $u,v,x,y \in V$ 
such that $x,y \in V_{u,v}$. 
Then $u*x=v$ and $u*y=v$, so $u*x=u*y$. 
Since $(V,*)$ is semi-smooth by (a), 
$u*x=u*y$ implies $u*(x*y)=u*x=v$ 
or $u*((x*y)*y)=u*x=v$ . 
Thus $x*y \in V_{u,v}$ 
or $x *^2 y = (x*y)*y \in V_{u,v}$. 

Second, we show that (b) implies (c). 
Take any $u,v,x,y,z,w \in V$ such that 
$x,y \in V_{u,v}$ and $y \in V_{x,z} \cap V_{z,w}$. 
Then $x*y=z$ and $z*y=w$. Therefore $w=(x*y)*y=x *^2 y$. 
By (b), we have $x*y \in V_{u,v}$ or $x *^2 y \in V_{u,v}$. 
Thus $z \in V_{u,v}$ or $w \in V_{u,v}$. 

Third, we show that (c) implies (a). 
Take any $u,x,y \in V$ 
such that $u*x=u*y$. 
Let $v:=u*x=u*y$. Then $x,y \in V_{u,v}$. 
Let $z:=x*y$ and $w:=z*y=(x*y)*y$. Then $y \in V_{x,z} \cap V_{z,w}$. 
By (c), we have $z \in V_{u,v}$ or $w \in V_{u,v}$. 
Therefore $u*z=v$ or $u*w=v$, 
that is, $u*(x*y)=u*x$ or $u*((x*y)*y)=u*x$. 
Thus $(V,*)$ is semi-smooth. 
\end{proof}

\begin{Defi}
A T-partition system $(V_{u,v} \mid (u,v) \in V \times V)$ on a graph $G$ 
is said to be \emph{semi-smooth} 
if the following condition holds: 
\begin{itemize}
\item[{\rm (R5)}]
For $u,v,x,y,z,w \in V$, 
if $x, y \in V_{u,v}$ and $x \in V_{y,z} \cap V_{z,w}$, 
then $z \in V_{u,v}$ or $w \in V_{u,v}$. 
\qed
\end{itemize}
\end{Defi}

\begin{Lem}\label{lem:SeSmVPS2SeSmTG}
Let $G$ be a graph, and 
let $\mathcal{P}$
be a semi-smooth T-partition system on $G$. 
Then, the groupoid associated with $\mathcal{P}$ 
is a semi-smooth travel groupoid on $G$. 
\end{Lem}

\begin{proof}
The lemma follows from Lemma \ref{lem:VPS2TG} 
and Proposition \ref{prop:semismooth-chara}. 
\end{proof}

\begin{Lem}\label{lem:SeSmTG2SeSmVPS}
Let $G$ be a graph, and 
let $(V,*)$ be a semi-smooth travel groupoid on $G$. 
Then, the right translation system 
of $(V,*)$ 
is a semi-smooth T-partition system on $G$. 
\end{Lem}

\begin{proof}
The lemma follows from Lemma \ref{lem:TG2VPS} 
and Proposition \ref{prop:semismooth-chara}. 
\end{proof}

\begin{Thm}\label{thm:chara-semismooth}
Let $G$ be a graph. 
Then, there exists a one-to-one correspondence between 
the set of all semi-smooth travel groupoids on $G$ 
and the set of all semi-smooth T-partition systems on $G$.  
\end{Thm}

\begin{proof}
The theorem follows from Theorem \ref{thm:main1} 
and Lemmas \ref{lem:SeSmVPS2SeSmTG} and \ref{lem:SeSmTG2SeSmVPS}. 
\end{proof}


\end{document}